\documentclass[a4paper,12pt]{amsart}

\usepackage{amsthm, amsmath, latexsym, amsfonts, epsfig}
\usepackage[all]{xy}
\usepackage{amssymb}
\usepackage{color}
\usepackage{tikz}

\newtheorem{Lemma}{Lemma}[section]
\newtheorem{Proposition}[Lemma]{Proposition}
\newtheorem{Theorem}[Lemma]{Theorem}

\newcommand{\Pic}{\mathrm{Pic}}

\newcommand{\Z}{\mathbb{Z}}

\newcommand{\Q}{\mathbb{Q}}

\newcommand{\bP}{\mathbb{P}}

\newcommand{\mgn}{\overline{\mathcal{M}}_{g,n}}

\newcommand{\sgnmint}{Pr_{g,n}^{\hspace{0.05cm}(m_1, \ldots, m_n)}}
\newcommand{\sgnm}{\overline{Pr}_{g,n}^{\hspace{0.05cm}(m_1, \ldots, m_n)}}

\title[Cohomology of moduli spaces of Prym curves]
{On the rational cohomology of \\ moduli spaces of Prym curves}
\author{Claudio Fontanari}

\email{claudio.fontanari@unitn.it}\curraddr{
{\sc Dipartimento di Matematica \\ Universit\`a degli Studi di Trento\\
Via Sommarive 14 \\ 38123 Trento \\ Italy.}}


\begin{document} 
\begin{abstract}
We investigate the low degree rational cohomology groups of the moduli space of twisted Prym curves ${\sgnm}$, where the integer twists $0\leq m_i\leq 1$ have even sum over $i$. We prove that these groups vanish in odd degree $\leq3$ and that the group in degree $2$ is algebraic. In particular, the results cover the classical moduli spaces of Prym curves and Prym curves with simple ramifications.
\end{abstract}

\maketitle

\section{Introduction}\label{sec0}

The first studies of low degree cohomology groups of the moduli space of curves $\overline{M}_{g,n}$ go back to the early nineties. A few years later Arbarello and Cornalba developed an algebro-geometric approach in their work \cite{ArbCor:98}, by exploiting the geometric description of the boundary in terms of moduli spaces of curves with smaller invariant $g$. This method is nowadays known as “Arbarello and Cornalba's induction" and it still stands as a very effective tool to this kind of investigation. For instance, their vanishing result in odd degree has been extended from degree $5$ up to degree $9$ and their approach has been fruitfully applied also in higher degree (see, e.g. \cite{BFP:24, CLP:23, CLPW:24}), by combining their original ideas with new techniques. 

A natural new direction to investigate is how to adapt their method to other moduli spaces of curves: this has been done, for instance, in degree $\leq 3$ for the moduli space of spin curves in \cite{BF:12} and in degree $1$ for the Hurwitz space with marked ramifications in \cite{L:25}.

In this note, we follow this line of thought and address the question for the moduli space of twisted Prym curves.  A {\em smooth twisted prym curve} is a smooth projective curve of genus $g$ over ${\mathbb C}$ together with the choice of $n$ marked points $p_1,\dots, p_n$ and a line bundle $L$ such that $L^{\otimes 2} \simeq \mathcal{O}_C(- \sum_{i=1}^n
m_i p_i )$, for some integers $0 \leq m_1,\dots , m_n \leq 1$ with even $\sum_{i=1}^n m_i$. For fixed $g$, $n$, $m_1,\cdots, m_n$, there is a moduli space ${\sgnmint}$ (see Section \ref{sec1}). These moduli spaces generalize two classical cases: when $i=0$ (or, more generally, all $m_i=0$) we recover the classical moduli space of Prym curves $\mathcal{R}_{g}$ (or, respectively, of $n$-marked Prym curves), and when all $m_i=1$ we recover the moduli space of Prym curves $\mathcal{R}_{g,2n}$ with ramifications at $2n$ prescribed points $p_i$. A natural compactification of these two moduli spaces has already been constructed in both cases: when all $m_i=0$ in \cite{BCF:04} and when all $m_i=1$ in \cite{Bud:24}. Following these previous approaches, in Section \ref{sec1} we introduce a new compactification ${\sgnm}$ for any choice of integers $m_i$ as above. The main feature of this compactification is that it admits a recursive description of the boundary, similar to those used for $\overline{M}_{g,n}$ in \cite{ArbCor:98} and $\overline{S}_{g,n}$ in \cite{BF:12}. That is, one can write the boundary in terms of the same kind of moduli spaces but with smaller invariants $g$ or $n$. The main contribution of this note is the study of low degree cohomology groups of ${\sgnm}$. Our results state that the moduli space ${\sgnm}$ has the same cohomological behaviour as $\overline{M}_{g,n}$ in degree at most $3$. 

More precisely, the odd cohomology in degree $\leq 3$ vanishes: 
\begin{Theorem}\label{prymvanishing}
For every $g$, $n$ and $(m_1, \ldots, m_n)$ as above, we have
$$
H^1(\sgnm, \Q) = H^3(\sgnm, \Q) = 0.
$$
\end{Theorem}

Moreover, the cohomology in degree $2$ is algebraic and coincides with the Picard group:

\begin{Theorem}\label{prymvH2}
For every $g$, $n$ and $(m_1, \ldots, m_n)$ as above, 
$$
H^2(\sgnm, \Q)  = \Pic(\sgnm) \otimes \Q.
$$
and is generated by the boundary classes of $\sgnm$, and the pullbacks of the Hodge class $\lambda$ and of the cotangent classes $\psi_i$, $i=1, \ldots, n$, along the natural map  $\sgnmint\to M_{g,n}$. 
\end{Theorem}

In particular, notice that these theorems cover the case where $i=0$ (or more generally all $m_i=0$), or in other words the classical moduli space of Prym curves $\mathcal{R}_{g}$ (or more generally of $n$-marked Prym curves) with the compactification constructed in \cite{BCF:04}. It is worth mentioning that one component of this compactification is isomorphic to the classical compactification via admissible covers provided in \cite{B:77}. They also cover the case where all $m_i=1$, that is the case of the moduli space of Prym curves $\mathcal{R}_{g,2n}$ with ramifications at $2n$ prescribed points $p_i$ with the compactification considered in \cite{Bud:24}.

To prove the first theorem we adapt the inductive approach introduced for
$\overline{M}_{g,n}$ in \cite{ArbCor:98} to ${\sgnm}$. To run the induction we make use of two ingredients: first we use a generalization to covers of $M_{g,n}$ of Harer's vanishing result on the compactly supported cohomology of $M_{g,n}$ proved in \cite{Harer:86}. This straightforward generalization has been proven in \cite{BF:12}. Secondly, we use the above mentioned recursive structure of the boundary of ${\sgnm}$. In such a way, the problem reduces to the study of some cohomology groups of a finite number of base cases. This is presented in Section \ref{sec3}. The proof of the second theorem relies again on the above mentioned generalization of Harer's vanishing result and on the study of the Chow group. This is contained in Section \ref{sec4}.

\medskip

{\bf Conventions.} Throughout, we work over the field ${\mathbb C}$ of complex numbers
and all cohomology groups are intended to be with rational
coefficients.
 
\vspace{0.2cm} {\bf Acknowledgments.} 
This research started during the
workshop “Cycles on moduli spaces” at the Centre International de Recontres Mathématiques (CIRM) in Luminy, 2025, when Carolina Tamborini and Sara Torelli raised the question whether the approach 
developed in \cite{BF:12} for moduli of spin curves might be adapted also to the Prym case. 
The author is grateful to the organizers of the conference for the stimulating atmosphere and to Carolina and Sara for many fruitful conversations. A special acknowledgement is due to Karl Christ for pointing out a gap in a previous version of this manuscript. The author is member of GNSAGA of INdAM (Italy).

\section{Moduli of twisted Prym curves}\label{sec1}
 
In this section we introduce the moduli space $\sgnmint$ of \emph{smooth twisted Prym curves}, together with its compactification $\sgnm$ of \emph{twisted Prym curves}. For that, let us fix non-negative integers $g,n,m_1,\dots,m_n$ such that
\begin{gather*}
    2g-2+n>0,\quad 0\le m_i\le 1\text{ for all }i,\quad \text{and}\quad \sum_{i=1}^n m_i \equiv 0 \pmod 2.
\end{gather*}

\noindent We define the moduli space of \emph{smooth twisted Prym curves}
\begin{eqnarray*}
\sgnmint &:=& \{ [(C, p_1, \ldots, p_n; L, \alpha]:
(C, p_1, \ldots, p_n) \textrm{ is a genus $g$} \\
& & \textrm{smooth projective curve with $n$ marked points}; \\
& & L \textrm{ is a line bundle of degree $-\frac{1}{2}  \sum_{i=1}^n
m_i$ on $C$ } \\ 
& &\textrm{with an isomorphism $\alpha: L^{\otimes 2}\simeq \mathcal{O}_C(- \sum_{i=1}^n
m_i p_i)$} \}. 
\end{eqnarray*}

Equivalently, a point of $\sgnmint$ determines a double cover $\widetilde C\to C$, branched exactly at the marked points $p_i$ with $m_i=1$; marked points with $m_i=0$ are \emph{not branch points}. The parity condition $\sum_i m_i \equiv 0 \pmod 2$ is precisely the Riemann--Hurwitz condition for the existence of such a cover.

We now define a compactification using \emph{quasi-stable curves}: these are semistable curves (i.e., nodal curves in which every rational component has at least two special points), such that rational components having exactly two special points (called \emph{exceptional components}) are pairwise disjoint.

\begin{eqnarray*}
\sgnm &:=& \{ [(C, p_1, \ldots, p_n; L; \alpha)]:
(C, p_1, \ldots, p_n) \textrm{ is a genus $g$} \\
& & \textrm{ quasi-stable projective curve with $n$ marked points}; \\
& &L \textrm{ is a line bundle of degree $-\frac{1}{2}  \sum_{i=1}^n
m_i$
on $C$ having} \\
& & \textrm{degree $1$ on every exceptional component of $C$,} \\
& &\textrm{and } \alpha: L^{\otimes 2} \to \mathcal{O}_C(- \sum_{i=1}^n m_i p_i)
\textrm{ is a homomorphism} \\
& & \textrm{which is nonzero at a general point of every}\\
& & \textrm{non-exceptional component of $C$} \}.
\end{eqnarray*}

The sets $\sgnmint$ and $\sgnm$ carry a natural structure of analytic variety as in \cite{BCF:04}, \S 1.3 (see also \cite{Bud:24}, proof of Proposition 2.5). 

It is well known that a compactification of the moduli space of Prym varieties can be given in terms of admissible covers (see e.g. \cite{B:77, Bud:24}). In this case one compactifies Prym data by viewing them as double coverings of stable curves and allowing suitable degenerations of the covering over the nodes. A key feature of the compactification $\sgnm$ introduced above is that 
it suggests a recursive description of the boundary in terms of moduli spaces with smaller invariants via gluing morphisms (in analogy with the boundary stratification of $\overline{M}_{g,n}$). 

More precisely, in order to make our inductive proof properly work, we should introduce, for every integer $k$ with $0 \le k \le n$ such that $n-k$ is even, the analogous compactification of the following moduli space:

\begin{eqnarray*}
{Pr}_{g,n}^{(m_1, \ldots, m_k, 0, \ldots, 0), k} &:=& \{ [(C, p_1, \ldots, p_n; L_1, 
\ldots, L_{\frac{n-k}{2}}, L)]:
(C, p_1, \ldots, p_n) \\
& & \textrm{is a genus $g$ projective curve with $n$ marked} \\
& & \textrm{points; $L_j$ is a line bundle of degree} -\frac{1}{2}  \sum_{i=1}^k m_i \\
& & \textrm{on the curve } C_j 
\textrm{ such that } L_j^{\otimes 2}\simeq \mathcal{O}_{C_j}(- \sum_{i=1}^k m_i p_i) \\ 
& & \textrm{and } 
L_{j+1} = \nu_j^*(L_j),
\textrm{ where we set } L_{\frac{n-k}{2}+1} :=L \\
& & \textrm{and } C_{\frac{n-k}{2}+1} := C, \textrm{ $C_j$ is obtained from } C_{j+1} \textrm{ by} \\
& & \textrm{identifying into an ordinary node the pair of} \\
& & \textrm{marked points } (p_{2n-k-2j}, p_{2n-k-2j-1}), \textrm{ and} \\
& & \nu_j: C_{j+1} \to C_j \textrm{ is the partial normalization map} \}. 
\end{eqnarray*}

By definition, if $\overline{Pr}_{g,n}^{(m_1, \ldots, m_k, 0, \ldots, 0), k}$ is the natural compactification 
of ${Pr}_{g,n}^{(m_1, \ldots, m_k, 0, \ldots, 0), k}$ defined in terms of quasi-stable curves, we have the boundary map
$
   \mu: {Pr}_{g-1,n+2}^{(m_1, \ldots, m_{k}, 0, \ldots, 0), k} \to
\overline{Pr}_{g,n}^{(m_1, \ldots, m_k, 0, \ldots, 0), k} 
$
sending

\noindent
$[(C, p_1, \ldots, p_n; L_1, \ldots, L_{\frac{n+2-k}{2}},L)]$ to
$[(C_{\frac{n+2-k}{2}}, p_1, \ldots, p_n; L_1, \ldots, L_{\frac{n-k}{2}},$ $L_{\frac{n+2-k}{2}})]$.

On the other hand, in order to make our paper more readable, we are going to freely identify the compactifications of ${Pr}_{g,n}^{(m_1, \ldots, m_k, 0, \ldots, 0), k}$ for every admissible integer $k$ with $\overline{Pr}_{g,n}^{(m_1, \ldots, m_n)}$, exactly as it was done in \cite{BF:12} in the completely analogous case of spin moduli spaces.

We stress anyway that, to get an inductive structure of the relevant moduli spaces, it is necessary to introduce twisted Prym curves, even when interested in compactifying the classical case of Prym curves. To make this more explicit, let us look for instance at the case when $n=0$ (hence, all $m_i=0$). Consider first a general point of the boundary divisor
\[
\Delta_i \subset \overline M_{g},
\]
parametrizing reducible curves $C=C_1\cup C_2$ meeting at a node. In this case, a Prym structure on $C$ is simply given by the choice of Prym structures on the two components $C_1$ and $C_2$. This behavior is compatible with a recursive description of the boundary: in fact, this yields boundary components of $\overline{\mathcal{P}}_g=\overline{Pr}_{g,0}$ which are image of the gluing morphisms
$$\mu_i : \overline{Pr}_{i,1}^{(0)}\times \overline{Pr}_{g-i,1}^{(0)}\rightarrow \overline{\mathcal{P}}_g.$$
In other words, we can picture the gluing map as follows:
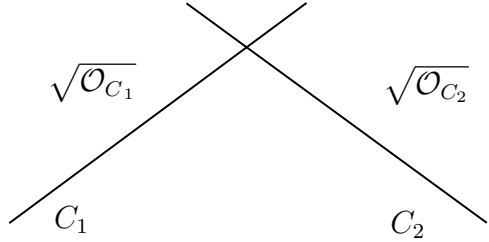
\begin{center}

\tikzset{every picture/.style={line width=0.75pt}} %set default line width to 0.75pt        

\begin{tikzpicture}[x=0.75pt,y=0.75pt,yscale=-1,xscale=1]

%Straight Lines [id:da09082034731920474] 
\draw    (119.83,90) -- (270.5,200) ;
%Straight Lines [id:da5904424017221581] 
\draw    (180.08,89.83) -- (30.92,200.17) ;

% Text Node
\draw (48.67,118.07) node [anchor=north west][inner sep=0.75pt]    {$\sqrt{\mathcal{O}_{C}{}_{_{1}}}$};
% Text Node
\draw (217,119.4) node [anchor=north west][inner sep=0.75pt]    {$\sqrt{\mathcal{O}_{C_{2}}}$};
% Text Node
\draw (52,190.07) node [anchor=north west][inner sep=0.75pt]    {$C_{1}$};
% Text Node
\draw (220.67,192.73) node [anchor=north west][inner sep=0.75pt]    {$C_{2}$};

\end{tikzpicture}

\end{center}

\noindent The situation is more complicated over the irreducible boundary divisor
$$
\Delta_{\mathrm{irr}}\subset \overline M_{g},
$$
whose general element is an irreducible nodal curve $C$ with only one node. Let
$$
\nu:Y\to C
$$
be its normalization, and let $p_1,p_2\in Y$ be the preimages of the node. There are two possible Prym structures on $C$. 

The first one is obtained by gluing the Prym datum on $Y$ via an identification of the fibers at $p_1$ and $p_2$, and this situation is compatible with a recursive description of the boundary. These Prym curves in fact lie in the image of the gluing morphism
$$\mu: {Pr}_{g-1,2}^{(0,0),2} \rightarrow \overline{\mathcal{P}}_g,$$
which we can illustrate as: 

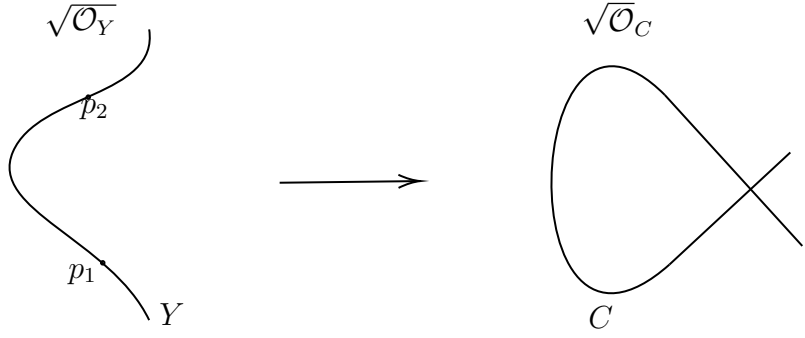
\begin{center}

\tikzset{every picture/.style={line width=0.75pt}}      

\begin{tikzpicture}[x=0.75pt,y=0.75pt,yscale=-1,xscale=1]

\draw    (374.17,211.5) .. controls (292.17,284.17) and (299.17,53.83) .. (371.83,126.5) ;
%Straight Lines [id:da6611478529800419] 
\draw    (371.83,126.5) -- (440.83,202.17) ;
%Straight Lines [id:da8358530111336734] 
\draw    (374.17,211.5) -- (434.83,155.33) ;
%Straight Lines [id:da5406435543054208] 
\draw    (178.63,170.5) -- (247.17,169.69) ;
\draw [shift={(249.17,169.67)}, rotate = 179.32] [color={rgb, 255:red, 0; green, 0; blue, 0 }  ][line width=0.75]    (10.93,-3.29) .. controls (6.95,-1.4) and (3.31,-0.3) .. (0,0) .. controls (3.31,0.3) and (6.95,1.4) .. (10.93,3.29)   ;
%Curve Lines [id:da07561863483290121] 
\draw    (113.17,239.33) .. controls (93.5,198.67) and (32.17,184.33) .. (44.83,155) .. controls (57.5,125.67) and (118.33,127.67) .. (113.17,93.67) ;
%Shape: Circle [id:dp6356783798103952] 
\draw   (89,210.63) .. controls (89,211.11) and (89.39,211.5) .. (89.88,211.5) .. controls (90.36,211.5) and (90.75,211.11) .. (90.75,210.63) .. controls (90.75,210.14) and (90.36,209.75) .. (89.88,209.75) .. controls (89.39,209.75) and (89,210.14) .. (89,210.63) -- cycle ;
%Shape: Circle [id:dp2861724906161758] 
\draw   (81.75,127.63) .. controls (81.75,127.14) and (82.14,126.75) .. (82.63,126.75) .. controls (83.11,126.75) and (83.5,127.14) .. (83.5,127.63) .. controls (83.5,128.11) and (83.11,128.5) .. (82.63,128.5) .. controls (82.14,128.5) and (81.75,128.11) .. (81.75,127.63) -- cycle ;

% Text Node
\draw (328.67,78.23) node [anchor=north west][inner sep=0.75pt]    {$\sqrt{\mathcal{O}}_{C}$};
% Text Node
\draw (71.25,209.65) node [anchor=north west][inner sep=0.75pt]    {$p_{1}$};
% Text Node
\draw (77.25,127.4) node [anchor=north west][inner sep=0.75pt]    {$p_{2}$};
% Text Node
\draw (116.79,229.48) node [anchor=north west][inner sep=0.75pt]    {$Y$};
% Text Node
\draw (332.13,230.73) node [anchor=north west][inner sep=0.75pt]    {$C$};
% Text Node
\draw (59.38,79.07) node [anchor=north west][inner sep=0.75pt]    {$\sqrt{\mathcal{O}_{Y}}$};

\end{tikzpicture}

\end{center}

Notice that there are two different ways to glue $\sqrt{\mathcal{O}_{Y}}$ on $Y$ to 
$\sqrt{\mathcal{O}}_{C}$ on $C$: this is precisely the reason why the domain of the 
gluing morphism is not ${Pr}_{g-1,2}^{(0,0)}$, but its \'etale double cover 
${Pr}_{g-1,2}^{(0,0),2}$, parametrizing also one choice of  
$\sqrt{\mathcal{O}}_{C}$ between the two square roots induced by $\sqrt{\mathcal{O}_{Y}}$. 

A second possibility is obtained starting from a square root $L$ of $\mathcal{O}_Y(-p_1-p_2)$ as follows: one replaces $C$ by a quasi-stable curve
$$
C'=Y\cup E,
$$
where $E\simeq \mathbf P^1$ is an exceptional component meeting $Y$ at $p_1$ and $p_2$. The line bundle $L$ extends uniquely to a line bundle $L'$ on $C'$ by setting
$$
L'|_Y=L, \qquad L'|_E=\mathcal O_E(1).
$$
Now $L'$ satisfies $(L')^{\otimes 2}\cong \mathcal O_{C'}$, so $(C',L')$ is a natural twisted Prym curve. Notice that every Prym structure of this second type is realized as the image of a gluing morphism

$$\mu: \overline{Pr}_{g-1,2}^{(1,1)} \rightarrow \overline{\mathcal{P}}_g.$$

\noindent We can illustrate this as:

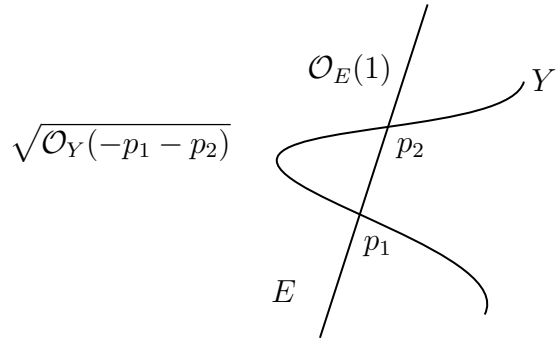
\begin{center}

\tikzset{every picture/.style={line width=0.75pt}} %set default line width to 0.75pt        

\begin{tikzpicture}[x=0.75pt,y=0.75pt,yscale=-1,xscale=1]

%Curve Lines [id:da5042172545140777] 
\draw    (371.67,194.83) .. controls (386.33,164.83) and (268,136.83) .. (267.33,117.83) .. controls (266.67,98.83) and (384.33,103.5) .. (391.33,78.17) ;
%Straight Lines [id:da600456030844549] 
\draw    (343,39.5) -- (289,206.5) ;

% Text Node
\draw (131.67,97.23) node [anchor=north west][inner sep=0.75pt]    {$\sqrt{\mathcal{O}_{Y}( -p_{1} -p_{2})}$};
% Text Node
\draw (309.36,154.38) node [anchor=north west][inner sep=0.75pt]  [rotate=-359]  {$p_{1}$};
% Text Node
\draw (326,105.23) node [anchor=north west][inner sep=0.75pt]    {$p_{2}$};
% Text Node
\draw (281.33,62.9) node [anchor=north west][inner sep=0.75pt]    {$\mathcal{O}_{E}( 1)$};
% Text Node
\draw (263,176.23) node [anchor=north west][inner sep=0.75pt]    {$E$};
% Text Node
\draw (393.67,70.23) node [anchor=north west][inner sep=0.75pt]    {$Y$};

\end{tikzpicture}

\end{center}

This case, hence, yields an obstruction to an inductive description of the boundary in terms of untwisted Prym data on lower genus curves.
This is precisely the reason why in order to get an inductive structure on the boundary one has to consider twisted Prym structures as well.\\

In the sequel, we will need the following.

\begin{Lemma}\label{lemI} The moduli spaces $\sgnmint$ and $\sgnm$ are the union of $2$ irreducible components if $m_i=0$ for every $i$, otherwise they are irreducible. 
\end{Lemma}

\begin{proof} 
It is enough to prove the statement for $\sgnm$. First of all note that if $m_i=0$ for every $i$, then the two irreducible components are cut out by the parity of the line bundle $L$ defining the Prym structure $L^{\otimes 2}\simeq \mathcal{O}_C $. Otherwise, there is some $j$ such that $m_j\neq 0$. In this case let $\pi: \sgnm \to \overline{Pr}_{g,m}^{\hspace{0.05cm}(1, \ldots, 1)}$ be the natural projection forgetting all 
marked points $p_i$'s such that $m_i=0$. 
By Proposition 2.5 in \cite{Bud:24}, $\overline{Pr}_{g,m}^{\hspace{0.05cm}(1, \ldots, 1)}$ is irreducible. Thus, the irreducibility of $\sgnm$ follows from the fact that $\pi$ is a proper surjective morphism with irreducible and equidimensional fibers.
\end{proof}

\section{Proof of Theorem \ref{prymvanishing}}\label{sec3}

In this Section we prove Theorem \ref{prymvanishing}. 
\subsection{Inductive approach.}\label{subI}  We adapt the inductive approach developed for $\overline{M}_{g,n}$ by Arbarello
and Cornalba in \cite{ArbCor:98} to the moduli spaces of twisted Prym curves $\sgnm$. Consider the long exact sequence of cohomology with compact supports:
\begin{equation}\label{exact}
\ldots \to H^k_c({Pr}_{g,n}^{\hspace{0.05cm}(m_1, \ldots, m_n)}) \to
H^k(\sgnm) \to H^k(\partial {Pr}_{g,n}^{\hspace{0.05cm}(m_1, \ldots, m_n)})
\to \ldots
\end{equation}
Whenever $H^k_c({Pr}_{g,n}^{\hspace{0.05cm}(m_1, \ldots, m_n)})
= 0$, there is an injection 
$$H^k(\sgnm) \hookrightarrow
H^k(\partial {Pr}_{g,n}^{\hspace{0.05cm}(m_1, \ldots, m_n)}).$$
Therefore, we can then exploit the recursive nature of the boundary, which is built as products of moduli spaces of twisted Prym curves with smaller invariant $g$. Indeed, each irreducible
component of the boundary of $\sgnm$ is essentially (see the detailed discussion is  Section \ref{sec1}) the image of a morphism:
$$
\mu_i: X_i \to \sgnm
$$
where either
$$
X_i = \overline{Pr}_{a, s+1}^{\hspace{0.05cm} (u_1, \ldots, u_{s+1})}
\times \overline{Pr}_{b, t+1}^{\hspace{0.05cm} (v_1, \ldots, v_{t+1})}
$$
with $a+b=g$, $s+t=n$, and $\sum_{i=1}^s u_i +\sum_{i=1}^t v_i =
\sum_{i=1}^n m_i$; or
$$
X_i = \overline{Pr}_{g-1, n+2}^{\hspace{0.05cm} (m_1, \ldots, m_n,
m_{n+1}, m_{n+2})}.
$$
Moreover, as in \cite[Lemma~2.6]{ArbCor:98}, 
\begin{equation}
H^k(\sgnm) \to \oplus_i H^k(X_i) \label{injective}
\end{equation}
is injective whenever $H^k(\sgnm) \to H^k(\partial
{Pr}_{g,n}^{\hspace{0.05cm}(m_1, \ldots, m_n)})$ is. \\

Putting all together, to prove the vanishing 
$$H^k(\sgnm)=0$$
of some odd $k$, one can apply the following inductive method.
Assume that $H^k_c({{Pr}}_{g,n}^{\hspace{0.05cm}(m_1, \ldots, m_n)})=0$ for all $k\leq c(g,d)$, where $c(g,n)$ is some non-negative integer depending on $g$ and $n$. Then 
$$H^k(\sgnm) \hookrightarrow \oplus_i H^k(X_i), \quad \text{for}\quad k\leq c(g,n),$$
where the $X_i$'s are moduli spaces of twisted Prym curves with strictly smaller genus.
In particular, if the vanishing $H^r(\overline{{Pr}}_{g',n'}^{\hspace{0.05cm}(m_1, \ldots, m_n)})=0$ for all odd $r\leq k$ holds for the pairs $(g',n')$ with $c(g',n')\leq k$, then the vanishing $H^k(\sgnm)=0$ holds for all $(g,n)$.

To prove Theorem~\ref{prymvanishing} using this inductive approach, it is therefore enough to show that:
\begin{enumerate}
\item $H^1_c({Pr}_{g,n}^{\hspace{0.05cm}(m_1, \ldots, m_n)}) =
H^3_c({Pr}_{g,n}^{\hspace{0.05cm}(m_1, \ldots, m_n)}) = 0$ for almost
all the values of $g$, $n$;
\item we check that $H^1(\sgnm) =
H^3(\sgnm) = 0$ for all the remaining values of $g$, $n$.
\end{enumerate} 

We show the first fact in Subsection \ref{subV} and the second one in Subsection \ref{subB}.

\subsection{Vanishing of compactly supported cohomology.}\label{subV} 
We rely on a generalization to finite étale connected covers of $M_{g,n}$ of Harer's vanishing result for $M_{g,n}$. This generalization is proved in \cite{BF:12}.  

More precisely, using the natural correspondence between level structures as in \cite[Section $2$]{BF:12} and finite étale connected covers of $M_{g,n}$, we give the equivalent formulation of \cite[Proposition $2.1$]{BF:12}. 

\begin{Proposition}\label{harer} Let $M\to M_{g,n}$ be a finite étale connected cover. Then
$H_k(M, {\mathbb Q})=0$ for $k > c(g,n)$, where
$$
c(g,n)= \left\{
\begin{array}{cc}
n-3 & g=0; \\
4g-5 & g > 0, n=0; \\
4g-4+n & g>0 , n > 0. \end{array}
\right.
$$
\end{Proposition}

Now notice that $\sgnm$ naturally arises as a
finite étale connected cover $\sgnm\to M_{g,n}$. We can then apply Proposition~\ref{harer} and by using Poincar\'e duality we deduce that
$$
H^1_c ({Pr}_{g,n}^{\hspace{0.05cm}(m_1, \ldots, m_n)}) = 0
$$
for any $g \ge 2$, for $g=1$, $n \ge 2$, and for $g=0$, $n \ge 5$;
and
$$
H^3_c ({Pr}_{g,n}^{\hspace{0.05cm}(m_1, \ldots, m_n)}) = 0
$$
for any $g \ge 3$, for $g=2$, $n \ge 2$, for $g=1$, $n \ge 4$, and
for $g=0$, $n \ge 7$.

\subsection{Base cases.}\label{subB}
 
By Subsections \ref{subI} and \ref{subV} we are left with the following base cases:
\begin{eqnarray}
\label{eq_one} H^1(\overline{Pr}_{0,n}^{\hspace{0.05cm}(m_1, \ldots, m_n)})
&=& 0,
\hspace{0.2cm} n \le 4 \\
\label{eq_two} H^3(\overline{Pr}_{0,n}^{\hspace{0.05cm}(m_1, \ldots, m_n)})
&=& 0, \hspace{0.2cm} n \le 6 \\
\label{eq_three} H^1(\overline{Pr}_{1,1}^{\hspace{0.05cm}(m)}) &=& 0 \\
\label{eq_four} H^3(\overline{Pr}_{2}) &=& 0 \\
\label{eq_five} H^3(\overline{Pr}_{2,1}^{\hspace{0.05cm}(m)}) &=& 0 \\
\label{eq_six} H^3(\overline{Pr}_{1,n}^{\hspace{0.05cm}(m_1, \ldots, m_n)})
&=& 0, \hspace{0.2cm} n \le 3.
\end{eqnarray}

We first address (\ref{eq_one}) and (\ref{eq_two}): on $\bP^1$, any two divisors of the same degree are linearly equivalent, which gives canonical identifications
\begin{equation}
\label{zero}
\overline{Pr}_{0,n}^{\hspace{0.05cm}(m_1, \ldots, m_n)}
\cong \overline{M}_{0,n}.
\end{equation}
Hence, \eqref{eq_one} and \eqref{eq_two} are immediate as Keel proved in \cite{Keel:92} that $H^k(\overline{M}_{0,n})=0$ for every odd $k$.

We now prove (\ref{eq_three}): first of all,
$\overline{Pr}_{1,1}^{\hspace{0.05cm}(1)} = \emptyset$
by degree reasons.
Next, Lemma \ref{lemI} shows that $\overline{Pr}_{1,1}^{\hspace{0.05cm}(0)}$ is
the union of two connected components

$$\overline{Pr}_{1,1}^{\hspace{0.05cm}(0)}= \overline{Pr}_{1,1}^{\hspace{0.05cm}(0), +} \sqcup \overline{Pr}_{1,1}^{\hspace{0.05cm}(0), -}.$$

The component 
$
\overline{Pr}_{1,1}^{\hspace{0.05cm}(0), -} $  is naturally isomorphic to  $\overline{M}_{1,1}
\cong \bP^1, 
$
so $H^1(\overline{Pr}_{1,1}^{\hspace{0.05cm}(0), -}) = H^1(\bP^1)=0$.
It remains to deal with $\overline{Pr}_{1,1}^{\hspace{0.05cm}(0), +}$.
To do so, observe that there exists a surjective morphism
$$
f: \overline{M}_{0,4} \longrightarrow
\overline{Pr}_{1,1}^{\hspace{0.05cm}(0), +}.
$$
which sends a $4$-pointed stable rational curve $(C; p_1, p_2, p_3, p_4)$ to the quasi-stable Prym curve $(E, q_1; L, \alpha)$ defined as follows. The quasi-stable curve $E$ is the double cover of $C$ branched at the $p_i$'s. The marked point is $q_1 \in E$ lying above $p_1$, and the line bundle is
$L = \mathcal{O}_E(q_1 - q_2),$ where $q_2 \in E$ lies above $p_2$. The homomorphism $\alpha: L^{\otimes 2} \rightarrow \mathcal{O}_E$ is induced by the relation $2(q_1 - q_2) \sim 0$ on $E$.
It follows that
$$
H^1(\overline{Pr}_{1,1}^{\hspace{0.05cm}(0), +}) \hookrightarrow
H^1(\overline{M}_{0,4}) = H^1(\bP^1) = 0
$$
which proves (\ref{eq_three}).

We now turn to \eqref{eq_four} and \eqref{eq_five}.
The space $\overline{Pr}_{2,1}^{\hspace{0.05cm}(1)}$ is empty by degree
reasons and 
$$
\overline{Pr}_{2,n}^{\hspace{0.05cm}(0, \ldots,0)}=\overline{Pr}_{2,n}^{\hspace{0.05cm}(0, \ldots, 0), +} \sqcup \overline{Pr}_{2,n}^{\hspace{0.05cm}(0, \ldots, 0), -}.
$$
The elements of the interior in each component are smooth hyperelliptic curves. If $C$ is a smooth hyperelliptic curve and $q_i$ ($i=1, \ldots, 6$)
are the ramification points of the hyperelliptic involution, then $C$ carries
one odd Prym structure (namely, $\mathcal{O}_C$) and $15$ even ones (namely, $\mathcal{O}_C(q_i-q_j)$, with $i$, $j$  distinct). We claim that there are surjective morphisms:
\begin{eqnarray*}
f^{\pm}: \overline{M}_{0,6} &\longrightarrow& \overline{Pr}_{2}^{\pm} \\
%f^-: \overline{M}_{0,6} &\longrightarrow& \overline{Pr}_{2}^- \\
g^{\pm}: \overline{M}_{0,7} &\longrightarrow&
\overline{Pr}_{2,1}^{\hspace{0.05cm}(0), {\pm}} 
%g^-: \overline{M}_{0,7} &\longrightarrow&
%\overline{Pr}_{2,1}^{\hspace{0.05cm}(0), -}
\end{eqnarray*}
In order to define $f^+$ and $f^-$, let $(C; p_1, \ldots p_6)$ be a
$6$-pointed, stable, genus zero curve. The morphism $f^+$ (respectively,
$f^-$) associates to it the admissible covering $Y$ of $C$ branched at the
$p_i$'s and equipped with the line bundle $\mathcal{O}_Y(q_1-q_2)$
(respectively, $\mathcal{O}_Y$),
where $q_i$ denotes the point of $Y$ lying above $p_i$.
As for $g^+$ and $g^-$, let $(C; p_1, \ldots, p_7)$ be a
$7$-pointed stable genus zero curve. The morphism $g^+$ (respectively,
$g^-$) associates to it the admissible covering $Y$ of $C$ branched at the
$p_i$'s ($i = 1, \ldots 6$), pointed at one of the two points lying above $p_7$ (of course
different choices produce isomorphic curves) and equipped with the line bundle
$\mathcal{O}_Y(q_1-q_2)$ (respectively, $\mathcal{O}_Y$),
where $q_i$ denotes the point of $Y$ lying above $p_i$.
Hence we obtain injective maps in cohomology
$H^k(\overline{Pr}_{2}^+)\hookrightarrow H^k(\overline{M}_{0,6})$,
$H^k(\overline{Pr}_{2}^-)\hookrightarrow H^k(\overline{M}_{0,6})$,
$H^k(\overline{Pr}_{2,1}^{\hspace{0.05cm}(0,0), +}) \hookrightarrow H^k(\overline{M}_{0,7})$,
and $H^k(\overline{Pr}_{2,1}^{\hspace{0.05cm}(0,0), -}) \hookrightarrow H^k(\overline{M}_{0,7})$. 
Thus, (\ref{eq_four}) and (\ref{eq_five}) follow by the above mentioned Keel's results (see \cite{Keel:92}).

To conclude we prove (\ref{eq_six}).  There is a natural isomorphism
\begin{equation}\label{eq:prs}
\overline{Pr}_{1,n}^{\hspace{0.05cm}(m_1, \ldots, m_n)} \cong \overline{S}_{1,n}^{\hspace{0.05cm}(m_1, \ldots, m_n)},
\end{equation}
where the second variety is the moduli space of twisted spin curves introduced in \cite{BF:12}. 
Indeed, if $(C,D) \in Pr_{1,n}^{\hspace{0.05cm}(m_1, \ldots, m_n)}$ then $K_C = \mathcal{O}_C$ and $2D \sim \mathcal{O}_C(-m_1 p_1 - \ldots - m_n p_m)$ implies $2(-D) \sim K_C + m_1 p_1 + \ldots + m_n p_m$, hence $(C,-D) \in S_{1,n}^{\hspace{0.05cm}(m_1, \ldots, m_n)}$.

As a consequence, (\ref{eq_six}) coincides with the corresponding statement for the moduli space of spin curves, which holds by \cite[Lemma 3.5]{BF:12}

\section{Proof of Theorem \ref{prymvH2}}\label{sec4}

In this Section we prove Theorem \ref{prymvH2}.

Exactly as in the case of $\mgn$, the injection (\ref{injective}) is compatible with the
Hodge decomposition. Hence, Proposition~\ref{harer} and
Poincar\'e duality imply that $H^{2,0}(\sgnm)=0$ for every $g,n$ if
$H^{2,0}(\overline{Pr}_{0,n}^{\hspace{0.05cm}(m_1, \ldots, m_n)})=0$,
$n \le 5$, and $H^{2,0}(\overline{Pr}_{1,n}^{\hspace{0.05cm}(m_1,
\ldots, m_n)})=0$, $n \le 2$. Notice now that
$H^2(\overline{Pr}_{0,n}^{\hspace{0.05cm}(m_1, \ldots, m_n)})$ is
algebraic by (\ref{zero}) and \cite{Keel:92}. 
Moreover, $H^2(\overline{Pr}_{1,n}^{\hspace{0.05cm}(m_1,
\ldots, m_n)})$, for $n \le 2$, is also algebraic (by \eqref{eq:prs} and \cite[Lemma 3.2]{BF:12}).
Thus $H^{2,0}(\sgnm)$ vanishes. 

From the
exponential sequence
$$
0 \to \Z \to \mathcal{O} \to \mathcal{O}^* \to 0
$$
we deduce that
$$
H^2(\sgnm) \cong H^1(\sgnm, \mathcal{O}^*) = \Pic(\sgnm) \otimes \Q.
$$
This concludes the proof of the first part of Theorem \ref{prymvH2}. 

Let us now provide a set of generators for $\Pic(\sgnm) \otimes \Q$. 
Since $\sgnm$ is normal,
there is an injection:
$$
\Pic(\sgnm) \hookrightarrow A_{3g-4+n}(\sgnm).
$$ Moreover, from the construction of $\sgnm$ it follows that the
singularities of $\sgnm$ are of finite quotient type, so every Weil
divisor is $\Q$-Cartier and there is a surjective morphism:
$$
\Pic(\sgnm) \otimes \Q \twoheadrightarrow A_{3g-4+n}(\sgnm) \otimes \Q.
$$
Hence there is a
natural isomorphism
$$ \Pic(\sgnm) \otimes \Q \stackrel{\cong}{\longrightarrow}
A_{3g-4+n}(\sgnm) \otimes \Q$$
and we may use the exact sequence
$$
A_{3g-4+n}(\sgnm \setminus Pr_{g,n}^{\hspace{0.05cm}(m_1, \ldots, m_n)}) \to $$
$$\to A_{3g-4+n}(\sgnm) \to A_{3g-4+n}(Pr_{g,n}^{\hspace{0.05cm}(m_1, \ldots, m_n)})
\to 0
$$
to conclude that $\Pic(\sgnm) \otimes \Q$ is generated by
the generators of $A_{3g-4+n}(Pr_{g,n}^{\hspace{0.05cm}(m_1, \ldots, m_n)})$ together with the set of boundary classes
of $\sgnm$. 
Finally, recall that $Pr_{g,n}^{\hspace{0.05cm}(m_1, \ldots, m_n)}\to M_{g,n}$ is a finite étale cover and one can see that it satisfies the assumptions of \cite[Theorem 2.1]{Put:12}. Namely, the subgroup $\Gamma$ of the mapping class group that defines such a cover naturally contains the Torelli group, the same way $Pr_{g,n}^{\hspace{0.05cm}(m_1, \ldots, m_n)}\to M_{g,n}$ is known to do, since the twist is irrelevant on the generators of the Torelli group. Therefore, for $g \ge 5$, $A_{3g-4+n}(Pr_{g,n}^{\hspace{0.05cm}(m_1, \ldots, m_n)})$
is freely generated by the pullbacks of the Hodge class $\lambda$ and of the cotangent classes $\psi_i$, $i=1, \ldots, n$, along the natural map  $\sgnmint\to M_{g,n}$, due to \cite[Theorem 2.1]{Put:12}.

\end{document}